\documentclass[arxiv,a4paper,11pt]{scrartcl}
\usepackage{subfiles}
\usepackage[bookmarks,pagebackref,colorlinks=true,linkcolor=blue,urlcolor=blue,citecolor=blue]{hyperref}
\usepackage[alphabetic]{amsrefs}

\usepackage[utf8]{inputenc}
\usepackage[english]{babel}
\usepackage{eulervm}
\usepackage{amssymb}
\usepackage{amsmath}
\usepackage{latexsym}
\usepackage{amsthm}
\usepackage{eucal}
\usepackage{bbm}
\usepackage{chngcntr}
\usepackage{apptools}
\usepackage{mathtools}
\usepackage{authblk}
\usepackage[pdflatex]{crop}
\usepackage{tikz}
\usepackage{tikz-cd}
\usetikzlibrary{3d}
\usepackage{pgfplots}
\usepackage{subcaption}
\usepackage{microtype}
\usepackage{enumitem}
\usepackage{graphicx}
\usepackage{mathrsfs}
\usepackage[colorinlistoftodos]{todonotes}
\usepackage{yfonts}
\usepackage{tabularx}
\usepackage{color}

\allowdisplaybreaks

\setkomafont{disposition}{\normalfont\bfseries}
\numberwithin{equation}{section}

\theoremstyle{definition}
\newtheorem{thm}{Theorem}[section]
\AtAppendix{\counterwithin{thm}{section}}
\newtheorem{defn}[thm]{Definition}
\newtheorem{lemma}[thm]{Lemma}

\newtheorem{cor}[thm]{Corollary}

\newenvironment{sketch}{\emph{Sketch of Proof.}}{\hfill$\square$}

\newtheorem{innercustomthm}{Theorem}
\newenvironment{introthm}[1]
  {\renewcommand\theinnercustomthm{#1}\innercustomthm}
  {\endinnercustomthm}

\theoremstyle{definition}

\newcommand{\C}{\mathbb{C}}
\newcommand{\R}{\mathbb{R}}

\let\originalleft\left
\let\originalright\right
\renewcommand{\left}{\mathopen{}\mathclose\bgroup\originalleft}
\renewcommand{\right}{\aftergroup\egroup\originalright}

\usepackage{setspace}

\newcommand{\auth}[0]{{Renan Assimos, Bal\'azs M\'ark B\'ek\'esi} and Giuseppe Gentile}
\newcommand{\tit}[0]{{Perturbed cone theorems for proper harmonic maps}}
\newcommand{\kw}[0]{{Harmonic maps, foliated maximum principle, convexity, horosphere theorem, cone theorem, Riemannian cone theorem.}}

\hypersetup{
pdfauthor={\auth},%
pdftitle={\tit},%
colorlinks, linktocpage=true, pdfstartpage=1, pdfstartview=FitV,%
breaklinks=true, pdfpagemode=UseNone, pageanchor=true, pdfpagemode=UseOutlines,% 
plainpages=false, bookmarksnumbered, bookmarksopen=true, bookmarksopenlevel=1,%
hypertexnames=true, pdfhighlight=/O,%
urlcolor=blue, linkcolor=blue, citecolor=blue, %}
}
\title{\tit}
\author{\auth\thanks{Correspondence: \href{mailto:renan.assimos@math.uni-hannover.de}{renan.assimos@math.uni-hannover.de}, \href{https://bmbekesi.github.io/ }{balazs.bekesi@math.uni-hannover.de}, \href{mailto:giuseppe.gentile@math.uni-hannover.de}{giuseppe.gentile@math.uni-hannover.de}}}
\affil{\small Leibniz University Hannover, Welfengarten 1, 30167 Hannover, Germany}
\date{}

\pgfplotsset{compat=1.16}
\begin{document}
%\subjclass[2000]{53C43,58E20,53C12,53C42,53A10}
\maketitle

\begin{abstract}
    Inspired by the halfspace theorem for minimal surfaces in $\R^3$ of Hoffman-Meeks \cite{HoffmanMeeks}, the halfspace theorem of Rodriguez-Rosenberg \cite{rodriguez1998half}, and the cone theorem of Omori \cite{omori}, we derive new non-existence results for proper harmonic maps into perturbed cones in $\R^n$, horospheres in $\mathbb{H}^n$ and also into perturbed Riemannian cones. The technical tool in use is an extension of the foliated maximum principle appearing in \cite{assimos2019geometry} to the non-compact setting.
\end{abstract}
\textbf{Subject classification:} 53C43, 58E20, 53C12, 53C42, 53A10\medskip\\
\textbf{Keywords: }{Harmonic maps, foliated maximum principle, convexity, horosphere theorem, cone theorem, Riemannian cone theorem.}
\begin{spacing}{0.8}
\tableofcontents
\end{spacing}
%\newpage

%%%%%%%%%%
\section{Introduction}
%%%%%%%%%%

The analysis of harmonic maps has a central role in Riemannian geometry. 
The reason for such a centrality is due to the fact that several interesting objects in Riemannian geometry are \textit{harmonic}. 
For instance, minimal submanifolds are immersed submanifolds so that the immersion $f:M\to (N,h)$ is harmonic with respect to the induced metric $g=f^*h$.
Geodesics are yet another prominent example of harmonic maps. Indeed, $\gamma :I\to (N,h)$ where $I$ is an interval in $\R$ is a geodesic if and only if $\gamma$ is a harmonic map. Similarly, one has closed geodesics as harmonic maps $\gamma :S^1\to (N,h)$. These are just a few examples showing that, knowing the existence of harmonic maps as well as their behaviour, allows for applications to several other problems in Riemannian geometry.
\medskip

One of the first systematic approaches to the constructions of harmonic maps between manifolds can be traced back to the groundbreaking work of Eells and Sampson in 1964 \cite{EellsSampson}.
Here, the authors describe a deformation method for producing harmonic maps between closed manifolds, the so-called harmonic map flow.
The harmonic map flow allows for the conclusion of several existence results for harmonic maps in interesting settings, the most classical one being the existence of harmonic maps into Riemannian manifolds with non-positive sectional curvature.
Once the existence of harmonic maps between Riemannian manifolds is given, a naturally arising question is the following:
\begin{quote}
    Which obstructions would prevent the existence of a harmonic map?
\end{quote}

To the best of our knowledge, the first ever attempt to this question, although with a slightly different flavour, is the analysis performed by Omori in 1967.
In \cite{omori} the author shows that the second fundamental form of an isometric immersion inside a non-degenerate cone in $\R^n$ is positive definite at some point.
Less than twenty years later, in \cite{Cone} the authors generalise Omori's work and gain a result about the non-existence of harmonic maps mapping entirely into a non-degenerate cone in $\R^n$, see \cite[corollary 4]{Cone}.

Related to the previously formulated question are several very interesting results and conjectures, mostly formulated in terms of minimal hypersurfaces. 
In 1966 Eugenio Calabi \cite{Calabi} proposed the following conjecture:
\paragraph{Calabi conjecture}
\begin{itemize}
    \item[(i)] Complete (not necessarily proper) minimal hypersurfaces in $\R^n$ are unbounded.
    \item[(ii)] Non-flat complete minimal hypersurfaces in $\R^n$ have unbounded projection in any codimension 2 hyperplane.
\end{itemize}

We will not dive into the details of the Calabi conjecture. 
Nonetheless, it is important to point out that the conjecture, in its full generality, has been proven wrong  by some counterexamples.
For $(i)$ a counterexample has been given by Nadirashvili in \cite{Nadirashvili} where he constructs a minimal hypersurface entirely contained in a sphere. The second claim has been proven wrong by the counterexample provided by Jorge and Xavier in \cite{jorgexavier}.
In particular, the above counterexamples show that the Calabi conjecture can not hold in full generality.
Later on, in \cite{ColdingMinicozzi} Colding and Minicozzi showed that the Calabi conjecture is true if one requires embeddedness. 
Moreover, their work displays a deep connection between the Calabi conjecture and properness. We take this as an inspiration for assuming our harmonic maps to be proper as well.
\medskip

As previously mentioned, in \cite{Cone} 
%In the paper \cite{Cone} cited beforehand, 
the authors prove a non-existence result for harmonic maps into non-degenerate cones in $\R^n$. 
How wide can such a cone be?
There are several answers to this question (see e.g \cite{mari}).
Nonetheless, we would like to recall a very interesting result in the theory of minimal surfaces due to Hoffman and Meeks.
\begin{introthm}{1 in \cite{HoffmanMeeks}}
A connected, proper, possibly branched, non-planar minimal surface $M$ in $\R^3$ is not contained in a halfspace. 
\end{introthm}
In particular, this shows that for a special class of harmonic maps (proper minimal) into $\R^3$ the cone can be as wide as the whole halfspace. Interestingly enough, Hoffman and Meeks' result is extremely powerful and holds only for $\R^n$ with $n=3$; indeed, one can construct counterexamples for $n>3$.
\medskip

It is important to notice that the Hoffman and Meeks' halfspace theorem as well as Calabi's conjecture relate to minimal submanifolds of codimension $1$, that is hypersurfaces.

Here is where our work finds its place; that is, providing answers to the following question:

\begin{quote}
    Can we characterize regions $R$ inside the target manifold which would prevent the existence of proper harmonic maps, possibly of higher codimension, mapping entirely inside them?
\end{quote}

\subsection{Statement of the main results}

Our work is based on the extension of the maximum principle in \cite{Sampson1978} and \cite{assimos2019geometry}, which we refer to as the \textit{foliated maximum principle}, which we now state.

%So our first main result is exactly the foliated maximum principle.

\begin{introthm}{\ref{thm_renanito}}[Foliated maximum principle]
Let $(N,h)$ be a complete Riemannian manifold, $\mathcal{F}$ a strictly convex foliation, and  $f:(M,g)\rightarrow (N,h)$ a non-constant proper harmonic map. Let $q=f(p)\in \mathcal{F}$ be a point inside the foliation. Then one of the following holds.
	\begin{enumerate}
		\item The image of $f$ leaves the foliation on the concave boundary $q^*\in \partial \mathcal{F}_{>q}$ within finite distance $d(q,q^*)<\infty$, where the metric $d$ is induced by the Riemannian metric $h|_\mathcal{F}$.
		\item There is a sequence $f(p_n) = q_n\in \mathcal{F}_{>q}$ with $\lim_{n\rightarrow \infty} d(q,q_n) = \infty$. In this case, $\mathcal{F}$ is necessarily unbounded.
	\end{enumerate}
\end{introthm}
We refer the reader to the definitions of the objects cited in the statement in section \ref{sect:2}.
\medskip

With this powerful tool at our disposal, we prove a plethora of results, such as nice obstructions for the existence of harmonic maps into $\R^n$.

\begin{introthm}{\ref{thm:perturbed_wedge}}[Perturbed cone theorem]
Let $C$ be a perturbed cone in $\R^n$. Then every proper harmonic map from a complete Riemannian manifold $(M,g)$ inside a cone region of $C$ is constant.
\end{introthm}
The notion of a perturbed cone in this paper is vastly more general than in the known literature, where only cones with solid angle smaller than $\pi$ are considered. Details are in Definition~\ref{def:eucl_wedge}, but we can motivate the reader with an illustrative example: The graph of $f:\R\longrightarrow \R$; $f(x)=\log(x+1)$ for $x\geq 0$ and $f(x)=0$ for $x<0$ will be a perturbed cone in $\R^2$ with our generalised definition, while the part of $\R^2$ above this graph is going to be called a cone region, whose convex hull is the upper halfspace.\medskip

By looking at the bare minimum needed to prove Theorem \ref{thm:perturbed_wedge}, we give a definition for a perturbed cone in generic Riemannian manifolds.
With such a definition at hand, we prove the following:
\begin{introthm}{\ref{thm:cone_theorem}}[Perturbed Riemannian cone theorem]
	Let $C_{\gamma,r}$ be a perturbed Riemannian cone inside the complete Riemannian manifold $(N,h)$. Then every proper harmonic maps into $C_{r,\gamma}\setminus \partial C_{r,\gamma}$ is constant.
\end{introthm}

The definition of such cone is given in Definition~\ref{defn:riemanniancone} and it is widely general. Assumptions on the curvature of the target manifold are very mild and when $N=\R^n$ our definition allows us to reconstruct several previously known cone theorems in Euclidean spaces.

%%%%%%%%
\subsection{Structure of the paper and notation}
%%%%%%%%%

Section~\ref{sect:2} reminds the reader of the definition of a harmonic map, and thereafter the section is devoted to the proof of Theorem~\ref{thm_renanito}. As a simple application, a Liouville-type property for $\R^n$ is obtained.
\medskip

Section~\ref{sect:wedges} is where we move to generalised definitions of \textit{perturbed cones} in the Euclidean space and some more elaborated applications of the foliated maximum principle~\ref{thm_renanito}. In subsection~\ref{PerturbedSec} we prove our second main result, Theorem~\ref{thm:perturbed_wedge}. In subsection~\ref{HoroSec} we employ the Theorem \ref{thm_renanito} to the case of the hyperbolic $n$-space $\mathbb{H}^n$ resulting into a Horosphere theorem~\ref{thm:horoshphere}. \medskip

Finally, in section~\ref{RiemannSec} we collect all the previously obtained results to formalize a definition of perturbed cones in a Riemannian manifold. In particular, here is where we prove our last main result, namely Theorem \ref{thm:cone_theorem}.

\paragraph{Notation}
From now on $(M,g)$ and $(N,h)$ will be complete connected Riemannian manifolds of dimension $m$ and $n$ respectively.
Moreover, $M$ is always considered to be the domain and $N$ the target of the smooth maps which we consider.

\paragraph{Acknowledgement} The authors would like to thank Bill Meeks and Laurent Hauswirth for their discussions and Laurent Mazet for pointing out useful references. This article is a part of the second author's PhD thesis.

\section{Strict convexity and the foliated maximum principle}\label{sect:2}
\subsection{Harmonic maps and foliations}
%\paragraph{Harmonic maps} 
A smooth map $f:M \rightarrow N$ is said to be \textbf{harmonic} if it is a critical point of the Dirichlet energy
$$ E[f] = \frac{1}{2}\int_M \|df\|^2 \, d\mathrm{Vol}_g. $$
Here, $\|\cdot\|$ represents the induced Hilbert-Schmidt norm on $TM^*\otimes f^*TN$.
An alternative characterization of the harmonicity of $f$ can be obtained through the first variation of the Dirichlet energy, which leads to the elliptic partial differential equation $\Delta\,f = 0$. Here, $\Delta\,f = \mathrm{tr}_g\nabla df$ is the tension field, where $\nabla$ denotes the induced connection on $TM^*\otimes f^*TN$ by the Levi-Civita connections of $g$ and $h$. For more details, see \cite{eellslemaire}.
\\
Some of the most important examples of harmonic maps include harmonic functions, harmonic forms, geodesics, totally geodesic immersion, minimal immersions, (anti)-holomorphic maps between K\"ahler manifolds and special Lagrangians. It is worth noting that any non-existence result proven for a harmonic map automatically holds for the aforementioned examples.

%\paragraph{Proper harmonic maps} 
By a \textbf{proper harmonic} map $f:M \rightarrow N$ we mean a harmonic map $f$ which is also  \textbf{proper} in the topological sense, that is, the pre-image of compacts are compact.\footnote{In some communities properness means that the image $f(M)$ has a compact intersection with compact subsets of $N$. We do not use this convention.}

Whenever $M$ is a compact manifold, properness of a harmonic map $f$ is automatic.
Clearly this is not true for non-compact manifolds, as an example one can consider a geodesic from an open interval $(a,b)$.

%\paragraph{Convexity}
Another important concept that plays a central role in our work is  \textbf{strict convexity} of a hypersurface $S\subseteq N$.
That is, the second fundamental form $A(X,Y) := (\nabla^N_XY)^\perp$ of $S$ is positive definite.
Taking a family of strictly convex hypersurfaces leads to the following definition.

\begin{defn}
	We define a \textbf{strictly convex foliation} in $N$ as an open, connected and oriented subset $\mathcal{F}\subseteq N$ which is a foliation  whose leaves are connected, embedded, and strictly convex hypersurfaces. Additionally, the foliation satisfies the \textbf{separating property}: every non-boundary leaf $\mathcal{L}$ separates the foliation into at least two connected components. 
\end{defn}

For each leaf $\mathcal{L}$ of a strictly convex foliation $\mathcal{F}$, the orientation determines the choice of the unit normal vector $N$ such that the second fundamental form $A(X,Y) = h(\nabla_XY,N)$ is positive definite. Thus, for any leaf we can talk about its convex and concave side.

\paragraph{Examples of foliations} The following are examples of strictly convex foliations.

    \begin{enumerate}
	\item In $\R^n$ define the annulus $A_{r,R} = \{x\in \mathbb{R}^n \, \mid \, r^2< \|x\|^2 < R^2\}$ for $r<R$. Take a foliation $\mathcal{F}_A$ whose leaves are spheres $S_\rho$ of radius $r< \rho < R$, with their interior representing the convex side.
	\item Let $E$ be a fixed half equator in $S^2$ with round metric, and $\Omega \subset S^2 \setminus E$ be an open set with $\operatorname{dist}(\partial\Omega, E)>0$ . Define a foliation $\mathcal{F}_{S^2\setminus E}$ on $\Omega$ using the boundaries of geodesic balls of radius $(\pi - \varepsilon)/2$, where $0<\varepsilon<\operatorname{dist}(\partial\Omega, E)$, as follows. The centres of these geodesic balls lie on the oriented great circle passing through the midpoint of $E$, its antipodal point, and  intersecting $E$ orthogonally. The leaves of the foliation $\mathcal{F}$ are the subset of the boundaries of these geodesic balls, whose points have positive inner product with the velocity vector of the great circle above. Each leaf is a connected, embedded and strictly convex hypersurface. The separating property clearly holds for the foliation. For more details, see \cite{assimos2020spherical}.
	
	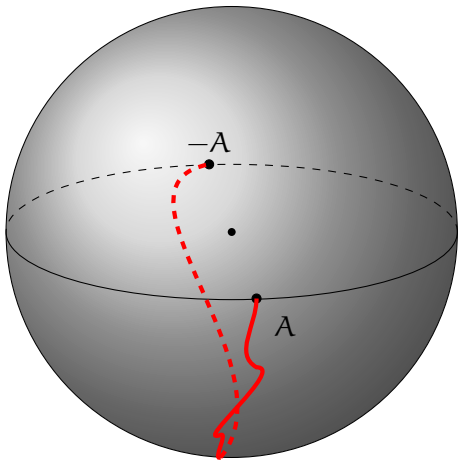
\begin{figure}[h]
		\centering
		\begin{tikzpicture}[scale=1.5]

			\shade[ball color = gray!40, opacity = 0.4] (0,0) circle (2cm);
			\draw (0,0) circle (2cm);
			\draw (-2,0) arc (180:360:2 and 0.6);
			\draw[dashed] (2,0) arc (0:180:2 and 0.6);
			\fill[fill=black] (0,0) circle (1pt);
			
			% Draw point A
			\draw[fill=black] (0.22,-0.59) circle (0.04) node[below right=0.1cm and 0.1cm] {$A$};
			
			% Draw point B
			\draw[fill=black] (-0.2,0.6) circle (0.04) node[above] {$-A$};
			
			%		% Draw the synodal curve
			%		\draw[red, ultra thick] (0.22,-0.59) to[out=-90, in=180] (0,-2);
			%		\draw[red, ultra thick, dashed] (0,-2) to[out=0, in=-140] (-0.2,0.6);
			
			% Calculate control points for curved sections
			\coordinate (cp1) at (0.22,-1.2);
			\coordinate (cp2) at (-0.1,-1.8);
			
			% Draw the wobbling synodal curve
			\draw[red, ultra thick] (0.22,-0.59)
			to[out=-90, in=150]
			(cp1) % First wobble point
			to[out=0, in=-160]
			(cp2) % Second wobble point
			to[out=20, in=-160]
			(-0.1,-2); % Third wobble point
			\draw[red, ultra thick, dashed] (-0.1,-2) to[out=50, in=-170]
			coordinate[pos=0.5] (midpoint) % Midpoint for control point calculation
			(-0.2,0.6);

		\end{tikzpicture}
		\caption{A perturbed half equator in $S^2$}\label{perturbedhalfequator}
	\end{figure}
	
	\item A perturbed version of example (b) can be constructed as follows, see Fig~\ref{perturbedhalfequator}. Let $\Gamma$ be a connected curve joining two antipodal points $A$ and $-A$ on $S^2$. Consider an open subset $\Omega$ of $S^2\setminus \Gamma$ with $\operatorname{dist}(\partial\Omega, \Gamma)>\varepsilon$ for some $\varepsilon>0$. We define a foliation $\mathcal{F}_{S^2\setminus \Gamma}$ on $\Omega$ using the boundaries of geodesic balls of radius $(\pi - \varepsilon)/2$ exactly like in the above example. The reader can notice that, as in the case where $\Gamma$ was a half-equator, the separation property also holds. This is because the foliation $\mathcal{F}_{S^2\setminus \Gamma}$ actually foliates $S^2\setminus (B_\varepsilon( A)\cup B_\varepsilon(-A))$, and as soon as we plug in a connected curve from $A$ to $-A$, the separation property automatically holds.
	
\end{enumerate}

\medskip

\paragraph{A partial order}
In the presence of a strictly convex foliation $\mathcal{F}$, one can introduce a relation $p<q$ for points $p,q\in \mathcal{F}$ if $p$ lies on the convex side of the leaf $\mathcal{L}_q$ passing through $q$. The relation $p\leq q$ denotes either $p<q$ or $p$ and $q$ lying in the same leaf. These relations do not define a (strict) partial order on $\mathcal{F}$ since $p \leq q$ and $q \leq p$ does not imply $p=q$. For $q\in \mathcal{F}$, we denote by $$ \mathcal{F}_{>q} = \bigcup_{q<r}\mathcal{L}_r $$ 
the concave side of $\mathcal{L}_q$ and by $\partial\mathcal{F}_{>q} := \partial\overline{(\cup_{q<r}\mathcal{L}_r)}\setminus \mathcal{L}_q$ the concave boundary of $\mathcal{F}_{>q}$.

\begin{figure}[htbp]
    \captionsetup[subfigure]{labelformat=empty}
	\centering
	\begin{subfigure}[b]{0.45\textwidth}
		\centering
		\begin{tikzpicture}[scale=1.25]
			%inner circles
			\draw [blue,thick,domain=0:360] plot ({0.25*cos(\x)}, {0.25*sin(\x)});
			\draw [black,domain=0:360] plot ({0.375*cos(\x)}, {0.375*sin(\x)});
			\draw [black,domain=0:360] plot ({0.5*cos(\x)}, {0.5*sin(\x)});
			\draw [black,domain=0:360] plot ({0.625*cos(\x)}, {0.625*sin(\x)});
			\draw [red,thick,domain=0:360] plot ({0.75*cos(\x)}, {0.75*sin(\x)});
			
			\draw [black,domain=-45:45] plot ({0.875*cos(\x)}, {0.875*sin(\x)});
			\draw [black,domain=-45:45] plot ({cos(\x)}, {sin(\x)});
			\draw [black,domain=-45:45] plot ({1.125*cos(\x)}, {1.125*sin(\x)});
			\draw [black,domain=-45:45] plot ({1.25*cos(\x)}, {1.25*sin(\x)});
			\draw [black,domain=-45:45] plot ({1.375*cos(\x)}, {1.375*sin(\x)});
			\draw [black,domain=-45:45] plot ({1.5*cos(\x)}, {1.5*sin(\x)});
			\draw [black,domain=-45:45] plot ({1.625*cos(\x)}, {1.625*sin(\x)});
			\draw [red,thick,domain=-45:45] plot ({1.75*cos(\x)}, {1.75*sin(\x)});

			\draw [black,domain=75:105] plot ({0.875*cos(\x)}, {0.875*sin(\x)});
			\draw [black,domain=75:105] plot ({cos(\x)}, {sin(\x)});
			\draw [black,domain=75:105] plot ({1.125*cos(\x)}, {1.125*sin(\x)});
			\draw [black,domain=75:105] plot ({1.25*cos(\x)}, {1.25*sin(\x)});
			\draw [black,domain=75:105] plot ({1.375*cos(\x)}, {1.375*sin(\x)});
			\draw [black,domain=75:105] plot ({1.5*cos(\x)}, {1.5*sin(\x)});
			\draw [black,domain=75:105] plot ({1.625*cos(\x)}, {1.625*sin(\x)});
			\draw [blue,thick,domain=75:105] plot ({1.75*cos(\x)}, {1.75*sin(\x)});

			\draw [black,domain=25:45] plot ({1.875*cos(\x)}, {1.875*sin(\x)});
			\draw [black,domain=25:45] plot ({2*cos(\x)}, {2*sin(\x)});
			\draw [black,domain=25:45] plot ({2.125*cos(\x)},{2.125*sin(\x)});
			\draw [black,domain=25:45] plot ({2.25*cos(\x)}, {2.25*sin(\x)});
			\draw [black,domain=25:45] plot ({2.375*cos(\x)}, {2.375*sin(\x)});
			\draw [blue,thick,domain=25:45] plot ({2.5*cos(\x)}, {2.5*sin(\x)});

			\draw [black,domain=-25:-45] plot ({1.875*cos(\x)}, {1.875*sin(\x)});
			\draw [black,domain=-25:-45] plot ({2*cos(\x)}, {2*sin(\x)});
			\draw [black,domain=-25:-45] plot ({2.125*cos(\x)},{2.125*sin(\x)});
			\draw [black,domain=-25:-45] plot ({2.25*cos(\x)}, {2.25*sin(\x)});
			\draw [black,domain=-25:-45] plot ({2.375*cos(\x)}, {2.375*sin(\x)});
			\draw [blue,thick,domain=-25:-45] plot ({2.5*cos(\x)}, {2.5*sin(\x)});
			
			\draw [black,domain=-15:15] plot ({1.875*cos(\x)}, {1.875*sin(\x)});
			\draw [black,domain=-15:15] plot ({2*cos(\x)}, {2*sin(\x)});
			\draw [black,domain=-15:15] plot ({2.125*cos(\x)},{2.125*sin(\x)});
			\draw [black,domain=-15:15] plot ({2.25*cos(\x)}, {2.25*sin(\x)});
			\draw [black,domain=-15:15] plot ({2.375*cos(\x)}, {2.375*sin(\x)});
			\draw [red,thick,domain=-15:15] plot ({2.5*cos(\x)}, {2.5*sin(\x)});
			
			\draw [black,domain=5:15] plot ({2.625*cos(\x)}, {2.625*sin(\x)});
			\draw [black,domain=5:15] plot ({2.75*cos(\x)}, {2.75*sin(\x)});
			\draw [black,domain=5:15] plot ({2.875*cos(\x)}, {2.875*sin(\x)});
			\draw [black,domain=5:15] plot ({3*cos(\x)}, {3*sin(\x)});
			\draw [black,domain=5:15] plot ({3.125*cos(\x)}, {3.125*sin(\x)});
			\draw [blue,thick,domain=5:15] plot ({3.25*cos(\x)}, {3.25*sin(\x)});
			
			\draw [black,domain=-15:-5] plot ({2.625*cos(\x)}, {2.625*sin(\x)});
			\draw [black,domain=-15:-5] plot ({2.75*cos(\x)}, {2.75*sin(\x)});
			\draw [black,domain=-15:-5] plot ({2.875*cos(\x)}, {2.875*sin(\x)});
			\draw [black,domain=-15:-5] plot ({3*cos(\x)}, {3*sin(\x)});
			\draw [black,domain=-15:-5] plot ({3.125*cos(\x)}, {3.125*sin(\x)});
			\draw [blue,thick,domain=-15:-5] plot ({3.25*cos(\x)}, {3.25*sin(\x)});
		\end{tikzpicture}
		\caption{A strictly convex foliation}
	\end{subfigure}
	\hfill
	\begin{subfigure}[b]{0.45\textwidth}
		\centering
		\begin{tikzpicture}[scale=1]
			% Vertices
			\node[circle, draw, fill=blue, inner sep=3pt] (v1) at (0,0) {};
			\node[circle, draw, fill=red, inner sep=3pt] (v2) at (1,0) {};
			\node[circle, draw, fill=red, inner sep=3pt] (v3) at (2,0) {};
			
			\node[circle, draw, fill=blue, inner sep=3pt] (v4) at (3,1.75) {};
			\node[circle, draw, fill=red, inner sep=3pt] (v5) at (3,0) {};
			\node[circle, draw, fill=blue, inner sep=3pt] (v6) at (3,-1.75) {};
			
			\node[circle, draw, fill=blue, inner sep=3pt] (v7) at (4,-1) {};
			\node[circle, draw, fill=blue, inner sep=3pt] (v8) at (4,1) {};
			\node[circle, draw, fill=blue, inner sep=3pt] (v9) at (1,1) {};
			
			% Directed edges
			\draw[->, thick] (v1) -- (v2);
			\draw[->, thick] (v2) -- (v3);
			\draw[->, thick] (v2) -- (v9);
			\draw[->, thick] (v3) -- (v4);
			\draw[->, thick] (v3) -- (v5);
			\draw[->, thick] (v3) -- (v6);
			\draw[->, thick] (v5) -- (v7);
			\draw[->, thick] (v5) -- (v8);
		\end{tikzpicture}
		\caption{The associated oriented graph}
	\end{subfigure}
	\caption{A foliation and its leaf space}\label{theduck}
\end{figure}

\paragraph{Leaf space}
    The notion of leaf space for a strictly convex foliation $\mathcal{F}$ can be viewed as a directed graph $G=(V,E)$, where the vertices $V$ correspond to
    
    \begin{enumerate}
        \item separating leaves, i.e. the leaves that produce at least three connected components due to the separating property or
        \item boundary components, i.e. for each separating leaf $\mathcal{L}$ as described in (a) insert a vertex for each component of $\mathcal{F}\setminus \mathcal{L}$ which does not contain a separating leaf.
    \end{enumerate}
    By inserting directed edges $E$ from convex to concave leaves successively, a directed graph is obtained. If there are no separating leaves, then the leaf space consists of two vertices connected by an oriented edge,
    
    The orders $<$ and $\leq$ on $\mathcal{F}$ descend to $G$ and define a strict partial order and a partial order on $G$ respectively. It is important to note that this graph does not contain any cycles; otherwise, the separating property would not hold for those leaves.

\subsection{The foliated maximum principle}

The maximum principle in \cite{assimos2019geometry} is a foliated version of the Sampson's maximum principle and the original result was proven in \cite{assimos2019geometry} for compact domains. 
Here we present an improved version for the non-compact case. 
First, let us recall the statement of the Sampson's maximum principle.

\begin{thm}[Sampson's maximum principle]\label{thm_sampson}
	Let $f:M\rightarrow N$ be a non-constant harmonic map, and let $S$ be a strictly convex hypersurface of $N$ passing through $q = f(p)$. Then for every open neighbourhood $U$ of $p$ in $M$, the image $f(U)$ cannot lie entirely on the convex side\footnote{Note that the convention of the convex side and the concave side are reversed in Sampson's paper compared to ours.} of $S$. 
\end{thm}

\begin{sketch}
	Let $u:V \rightarrow \R$ be a convex function on the open subset $V\subseteq N$ such that $q \in u^{-1}(0)=S\cap V$. The preimage of $(-\infty,0)$ denotes the convex side of $S$ and the preimage of $(0,\infty)$ its concave side. The composition formula for the Laplacian (see, proposition 2.20 in \cite{eellslemaire}) implies $$\Delta_M(u \circ f) = du(\Delta f) + \mathrm{tr}_g\nabla du(df,df) = \mathrm{tr}_g\nabla du(df,df) \geq0$$ where the last inequality follows by the strict convexity of $S$. Suppose that every neighbourhood of $p$ is mapped to the convex side of $S$, i.e. the inequality $\Delta (u\circ f)<0$ is satisfied.
	If $df_p\neq 0$, we obtain that $\Delta_M(u \circ f) > 0$ by the composition formula, thus contradicting the classical maximum principle. For the case $df_p=0$, we refer the reader to \cite{Sampson1978}.
\end{sketch}

\begin{thm}[Foliated maximum principle]\label{thm_renanito}
	Let $(N,h)$ be a complete Riemannian manifold, $\mathcal{F}$ a strictly convex foliation, and  $f:(M,g)\rightarrow (N,h)$ a non-constant proper harmonic map. Let $q=f(p)\in \mathcal{F}$ be a point inside the foliation. Then either
	\begin{enumerate}
		\item the image of $f$ leaves the foliation on the concave boundary $q^*\in \partial \mathcal{F}_{>q}$ within finite distance $d(q,q^*)<\infty$, where the metric $d$ is induced by the Riemannian metric $h|_\mathcal{F}$ or
		\item there is a sequence $f(p_n) = q_n\in \mathcal{F}_{>q}$ with $\lim_{n\rightarrow \infty} d(q,q_n) = \infty$. In this case, $\mathcal{F}$ is necessarily unbounded.
	\end{enumerate}
\end{thm}

\begin{proof}	
	The proof is an adaptation of Theorem 3.3 in \cite{assimos2019geometry}. Let $q=f(p)$ be as stated. Since $q$ lies on the leaf $\mathcal{L}_q$ passing through $q$, we can apply the Sampson maximum principle, which allows us to choose $q_1 = f(p_1)$ on the concave side of $\mathcal{L}_q$. We repeat this process of applying Sampson's theorem to each $q_i:=f(p_i)$ to find a point $q_{i+1}$ in the concave side of $\mathcal{L}_{q_i}$. Thus, we obtain two sequences of points $\{p_k\}$ in $M$ and $\{q_k\}$ in $\mathcal{F}$ . 
	\begin{itemize}
		\item Assume that $\{q_k\}$ accumulates to a point $q^*$.
		Due to the completeness of $N$ we conclude that $\{q_k\}$ converges to $q^*\in N$.
		Due to the properness of $f$, we conclude that $\{p_k\}$ is a sequence contained in the compact set $f^{-1}(\{q_k\}\cup \{q^*\})$.
		Due to compactness, one can consider a convergent subsequence $\{p'_k\}$ of $\{p_k\}$ converging to some $p^*$.
		Finally, continuity of $f$ leads to $f(p^*)=q^*$.
		Let now $\mathcal{L}_k$ be the leaf containing $q_k$. The constructed sequence satisfies $\mathcal{L}_1<\mathcal{L}_2<\dots$.
		Set $$ \mathcal{L}^*=\inf\{\mathcal{L}\, \mid\, \mathcal{L}_k<\mathcal{L} \text{ for all }k \in \mathbb{N}, \mathrm{Im}\, f \cap \mathcal{L}\neq \emptyset\},$$ where the infimum is taken with respect to the partial order $\leq$. Notice that $q^*\in \mathcal{L}^*.$ If $q^*\notin \partial\mathcal{F}_{>q}$, we get a contradiction by another application of the Sampson's maximum principle. This implies (a).
		\item Assume that any sequence $\{q_k\}$ constructed via the iterative method above is not convergent. 
		Then the sequence $\{q_k\}$ is divergent and monotone with respect to the order $<$. Thus, the only possibility is $\lim_{n\rightarrow \infty} d(q,q_n) = \infty$.
	\end{itemize}
\end{proof}

\paragraph{Comments on the foliated maximum principle}
Firstly, (a) does not automatically mean that $f$ has to leave the foliation. Only a part of the harmonic map might leave and could even re-enter the foliation later.

Secondly, the completeness of $N$ and properness of $f$ are essential. By $N$ being non-complete, the limit $q^*$ of the constructed sequence might not exist. An example violating properness would be a small piece of a unit-speed geodesic $\gamma:(-\varepsilon,\varepsilon)\rightarrow \mathbb{R}^2$, which can be included in any strictly convex foliation $\mathcal{F}$. A more involved non-proper example is given by the work of Nadirashvili in \cite{Nadirashvili}, where he constructs a non-proper bounded complete minimal disk in $\mathbb{R}^3$.

Lastly, the foliated maximum principle is valid in the context of harmonic sections $s$ of fibre bundles $F\rightarrow M$. Here, the vertical Dirichlet energy $E[s]:= \int_M \|D^Vs\|d\mathrm{Vol}_g$ with $D^Vs$ being the vertical projection of the differential $D$ is used for the definition of harmonicity. In this case, the strictly convex foliation only needs to foliate the vertical directions.

\paragraph{Liouville's theorem} 
As a first illustration of the non-compact foliated maximum principle, we prove a version of Liouville's theorem.

\begin{cor}[Liouville's theorem]\label{cor_liouville}
	The only proper harmonic maps from a complete (not necessarily compact) Riemannian manifold $(M,g)$ into a bounded open subset $B$ of $\mathbb{R}^n$ for $n>1$ are the constant maps. In particular, there are no bounded, proper minimal surfaces in $\mathbb{R}^n$.
\end{cor}
\begin{proof}
	By translating $B$, we can assume that $r = \inf_{q\in B}\|q\|>0$. Let $R = \sup_{q\in B}\|q\|$ after a possible translation. Then, $B \subseteq A_{r,R} = \{x\in \mathbb{R}^n \, \mid \, r^2\leq \|x\|^2 \leq R^2\}$, and Theorem \ref{thm_renanito} (a) applies.
\end{proof}

\section{Perturbed cone and Horosphere theorems}\label{sect:wedges}
In this section, we will prove the first two non-existence theorems for proper harmonic maps into specific regions of the Euclidean space and the hyperbolic space.
\subsection{Perturbed cone theorems in $\R^n$}\label{PerturbedSec}
We consider the Euclidean space $\R^n$ with its standard flat metric defined by $ds^2 = dx_1^2 + \dots + dx_n^2$.
Recall furthermore that geodesics in $\R^n$ with the flat metric are given by affine lines, i.e. $\gamma(t) = p + tv$ for $p,v \in \R^n$.

\paragraph{Perturbed cones in Euclidean spaces} The basic example for the definition of a perturbed cone is going to be the classical cone $C$ defined by the graph of $f:\R \rightarrow \R; \, x \mapsto |x|$ inside $\R^2$, henceforth called classical cones. The connected component $R$ of $\R^2 \setminus C$ containing $(0,1)$ has a special property: any $p\in R$ can be enclosed by the cone and a hyperplane, in this case a line. 
In particular, there are no affine lines in $R$.
Such a property seems to be very interesting and, most importantly, crucial for the analysis of harmonic maps.
Thus, we give an appropriate definition for a set possessing such \textit{enclosing property}.

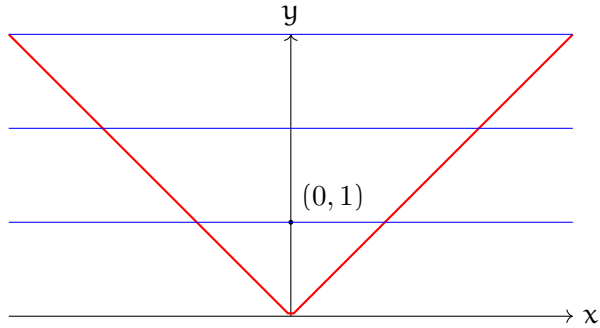
\begin{figure}[htbp]
	\centering
	\begin{tikzpicture}[scale=1.25]
		% Axes
		\draw[->] (-3,0) -- (3,0) node[right]{$x$};
		\draw[->] (0,0) -- (0,3) node[above]{$y$};
		%		\draw[->] (0,0,-3) -- (0,0,3) node[right]{$z$}
		% Graph
		\draw[fill=black] (0,1) circle (0.02) node[above right]{$(0,1)$};
		\draw[red,thick, domain=-3:3, samples=100] plot (\x, {abs(\x)});
		\draw[blue, domain=-3:3, samples=100] plot (\x, {1});
		\draw[blue, domain=-3:3, samples=100] plot (\x, {2});
		\draw[blue, domain=-3:3, samples=100] plot (\x, {3});
	\end{tikzpicture}
	\caption{The graph of $f(x)=|x|$ and enclosing hyperplanes}
	\label{fig:cone}
\end{figure}

\begin{defn}\label{defn:enclosing}
	Let $R\subseteq \R^n$ be a connected open set. $R$ is said to possess the \textbf{enclosing property} if, for every $p\in R$, there is an affine hyperplane $H$ such that the connected component $B$ of $R\setminus H$ containing $p$ is precompact, i.e. $\bar B$ is compact. 
\end{defn}

The blue horizontal lines in Figure \ref{fig:cone} verify that the enclosing property holds for the cone $C$. A non-example is the upper halfspace $\mathbb{H}^2 = \{ (x,y) \in \R^2 \,\mid \, y >0\}$. The enclosing property allows us to extend the notion of a cone to one of a \textit{perturbed cone}. 

\begin{defn}\label{def:eucl_wedge}
Let $C$ be a closed, path-connected subset of $\R^n$ such that $\R^n\setminus C$ consists of at least two connected components. Let $R$ be one of the connected components of $\R^n\setminus C$. If $R$ satisfies the enclosing property, then $C$ is called a \textbf{perturbed cone} in Euclidean space and $R$ a \textbf{cone region}.
\end{defn}

\begin{lemma}
    Let $C$ be a perturbed cone and $R$ a cone region. Then $R$ cannot contain affine lines.
\end{lemma}
\begin{proof}
    Assume there is an affine line $L$ in $R$. Let $p\in L$ and choose a separating hyperplane $H$ and the region $B$ via the enclosing property of $R$. If $L$ is parallel to $H$ then it would be fully contained in $B$ and hence contradicting the precompactness of $B$. If $L$ intersects $H$ then an unbounded part of $L$ would be in $B$ again contradicting the precompactness of $B$. Since $L$ can either be parallel to $H$ or intersect it, we exhausted all the possibilities.
\end{proof}
The enclosing property can be weakened in such a way that the previous lemma does not hold.
In this case, the new notion of a perturbed cone needs the assumption of \textit{not containing affine lines}.

\paragraph{Examples of perturbed cones}

The following are examples of perturbed cones in the Euclidean space.

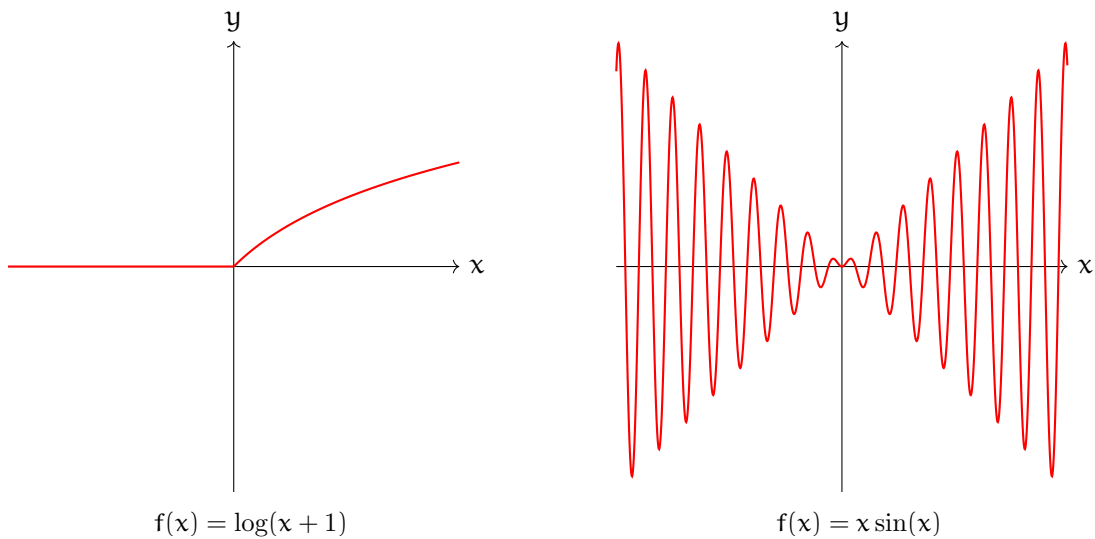
\begin{figure}[htbp]
    \captionsetup[subfigure]{labelformat=empty}
	\centering
	\begin{subfigure}[b]{0.45\textwidth}
		\centering
		\begin{tikzpicture}[scale=1]
			% Axes
			\draw[->] (-3,0) -- (3,0) node[right]{$x$};
			\draw[->] (0,-3) -- (0,3) node[above]{$y$};
			%		\draw[->] (0,0,-3) -- (0,0,3) node[right]{$z$}
			% Graph
			\draw[red,thick, domain=-3:0, samples=100] plot (\x, {0});
			\draw[red,thick, domain=0:3, samples=100] plot (\x, {ln(\x+1)});
		\end{tikzpicture}
		\caption{$f(x)=\log(x+1)$}
	\end{subfigure}
	\hfill
	\begin{subfigure}[b]{0.45\textwidth}
		\centering
		\begin{tikzpicture}[scale=1]
			% Axes
			\draw[->] (-3,0) -- (3,0) node[right]{$x$};
			\draw[->] (0,-3) -- (0,3) node[above]{$y$};
			%		\draw[->] (0,0,-3) -- (0,0,3) node[right]{$z$}
			% Graph
			\draw[red,thick, domain=-3:3, samples=1000] plot (\x, {\x * sin(1000*\x)});

		\end{tikzpicture}
		\caption{$f(x)=x\sin(x)$}
	\end{subfigure}
	\caption{Examples of perturbed cones}
\end{figure}

\begin{enumerate}[label=(\alph*)]
\item Let $C\subseteq \R^2\cong \C$ be defined by the union of the negative $x$-axis and the line $e^{i\theta} \R_{\geq 0}$ for some angle $\theta \in (0,\pi)$. Notice that for $\theta = 0$ this is not a perturbed cone, since there are geodesics in the upper half-plane.
\item A family of perturbed cones in $\R^2$ can be defined by glueing the negative $x$-axis and the graph of any continuous, injective, unbounded function $f:[0,\infty) \longrightarrow [0,\infty)$ with $f(0)=0$. For example, $f(x)=\log(x+1)$. This can be seen as a perturbation of the cone defined in (a). Note that the example using the logarithm is not contained in any classical cone with an angle smaller than $\pi$, since the convex hall of the graph is the upper halfspace.
\item The graph of $f(x)= x\sin(x)$ in $\R^2$ is an example such that both components of $\R^2\setminus \mathrm{gr}\,(f)$ are cone regions.
\item Let $p,v \in \R^n$ with $v \neq 0$ and an angle $0<\theta<\pi/2$. Then the boundary of the cone $$ C(p,v,\theta) = \left\{ p + x \, \mid \, x \in \R^n, |\langle x,v \rangle| \leq \cos(\theta) \|x\|\|v\| \right\} $$ obviously defines a perturbed cone with its inside being the cone region.
\item Let $K$ be a compact set with $K^\circ = K \setminus \partial K \neq \emptyset$. Then $K$ is a perturbed cone with cone region $K^\circ$.
\end{enumerate}

Next, we prove the main theorem of this subsection.

\begin{thm}[Perturbed cone theorem]\label{thm:perturbed_wedge}
Let $C$ be a perturbed cone in $\R^n$. Then every proper harmonic map from a complete Riemannian manifold $(M,g)$ inside a cone region of $C$ is constant.
\end{thm}
\begin{proof}
Let $f:M\rightarrow \R^n$ be a non-constant proper harmonic map and a point $p\in M$ chosen such that $q=f(p)$ lies in a cone region of $C$. Take now the associated enclosing compact set $B$ and the hyperplane $H$ such that $q\in B^\circ$.

\begin{minipage}[b]{0.55\textwidth}
Let $\nu$ be the unit normal with respect to $H$ pointing inwards to $B^\circ$ and the half-sphere $$ S = S_r^{n-1} \cap \{x\in \R^n\,\mid\, 0\leq \langle x,\nu \rangle\} $$ where $S_r^{n-1}$ is the sphere of radius $r = \mathrm{diam}(B)$. Then for $\gamma:(-r-\varepsilon,r+\varepsilon) \rightarrow \R^n$ defined by $\gamma(t)=q + t\, \nu$ for $\varepsilon>0$, we obtain an associated strictly convex foliation $$\mathcal{F} = \bigcup_{t \in (-r-\varepsilon,r+\varepsilon)}\mathcal{L}_{\gamma(t)}$$ with leafs $\mathcal{L}_{\gamma(t)}=\gamma(t)+S$. 
\end{minipage}
\begin{minipage}[b]{0.4\textwidth}
    \centering
    \begin{tikzpicture}[scale=1]
			\draw [blue,dotted,thick,domain=0:180] plot ({2*cos(\x+215)-1}, {2*sin(\x+215)+3.75});
			\draw [blue,dotted,thick,domain=0:180] plot ({2*cos(\x+215)-0.666}, {2*sin(\x+215)+3.25});
			\draw [blue,dotted,thick,domain=0:180] plot ({2*cos(\x+215)-0.333}, {2*sin(\x+215)+2.75});
			\draw [blue,dotted,thick,domain=0:180] plot ({2*cos(\x+215)}, {2*sin(\x+215)+2.25});
			\draw [blue,dotted,thick,domain=0:180] plot ({2*cos(\x+215)+0.333}, {2*sin(\x+215)+1.75});
			\draw [blue,dotted,thick,domain=0:180] plot ({2*cos(\x+215)+0.666}, {2*sin(\x+215)+1.25});
			\draw [-,thick] (-2, 2) -- (2, 0);
			\draw [-,thick] (1, 4) -- (2, 0) node[right]{$C$};
			\draw [-,dashed] (-2, 2) -- (1, 4);
			\draw[fill=black] (-1,3) circle (0.0) node[above]{$H$};
			\draw[fill=black] (0,2) circle (0.04) node[above]{$q$};
			\draw [->,thick] (0, 2) -- (0.333, 1.5) node[right]{$\nu$};
	\end{tikzpicture}

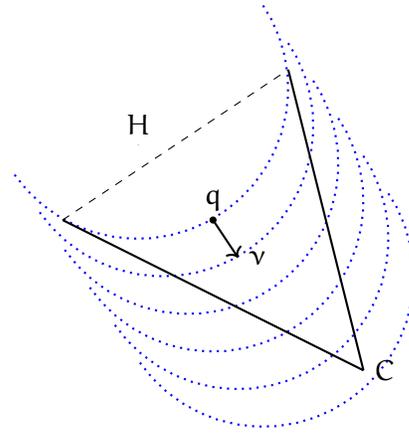
\captionof{figure}{The constructed foliation}
\end{minipage}\\
By definition, $q\in \mathcal{L}_{\gamma(-r)}$ and $\partial B \cap H$ lies on the convex side of $\mathcal{L}_{\gamma(-r)}$. Hence, by Theorem \ref{thm_renanito} (a) we obtain a contradiction of $\mathrm{Im}\, f$ being contained in a cone region of $C$.
\end{proof}

\paragraph{Local cones}

A corollary derived from the proof of our main theorem, which can be interpreted as a local manifestation of a cone theorem, goes as follows. 

\begin{figure}[h] \label{Fig:localwedpic}
	\centering 
	\begin{tikzpicture}[scale=2]
		\draw[-,thick] (-2,1) -- (-1,1);
		\draw[-,thick] (-1,1) -- (0,0.25);
		\draw[-,thick] (0,0.25) -- (1,1);
		\draw[-,thick] (1,1) -- (2,1) node[right]{$y=c$};
		\draw[-,dashed ,thick] (-0.93,1) -- (0.94,1);
		\draw[fill=black] (-0.3,0.8) circle (0.02) node[right]{$q\in f(M)$};
		% Axes
		\draw[<->] (-2,0) -- (2,0) node[below right=0.08cm and -1cm]{$x$};
	\end{tikzpicture}
	\caption{A local cone}
\end{figure}
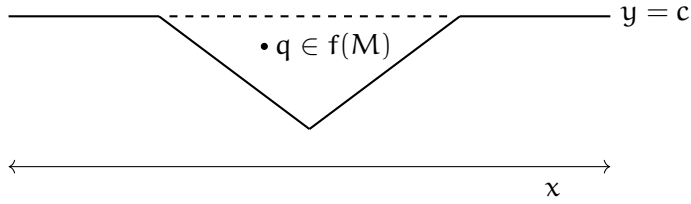

\begin{defn}
    A \textbf{local cone} $C$ in $\R^n$ is a closed subset such that there is a hyperplane $H$ and a precompact connected component $B$ of $\R^n\setminus(C \cup H)$. $B$ is called the \textbf{local cone region}.
\end{defn}
Notice that a local cone can always be extended to a cone in several non-unique ways.

\begin{cor}[Local cone theorem]
		Let $C$ be a local cone with cone region $B$. Then a non-constant proper harmonic map $f:M \rightarrow \R^n$ with $f(p) \in B$ also possesses a point $q\in M$ such that $f(q)\in C$. In other words, a harmonic map entering $B$ has to leave through $C$.
\end{cor}
\begin{proof}
    Since $B$ is an open set and $q \in B$, we can just repeat the construction in the proof of Theorem \ref{thm:perturbed_wedge}.
\end{proof}

\subsection{The horosphere theorem}\label{HoroSec}

To further demonstrate the flexibility of our methods on obtaining perturbed cone theorems, we proceed to explore hyperbolic upper halfspaces. We exhibit a 1-parameter family of hypersurfaces that foliates the space above a horosphere.  More precisely and with the nomenclature of our paper, we consider the upper halfspace model of the hyperbolic space $\mathbb{H}^{n+1} = \{ x \in \R^{n+1} \, \mid \, x_{n+1} > 0 \}$ equipped with the metric $ds^2 = \tfrac{1}{x_{n+1}^2}(dx_1^2 + \dots + dx_{n+1}^2)$ of constant sectional curvature $-1$.
We will show that a cone in the region $\{x_{n+1}>c>0\}$ in the hyperbolic space $\mathbb{H}^{n+1}$ can be as wide as a halfspace; that is a cone with angle $\pi$.

\paragraph{Geometry of graphical hypersurfaces in $\mathbb{H}^{n+1}$} We shall briefly discuss the geometry of hypersurfaces of $\mathbb{H}^{n+1}$. Let $F:\R^n\rightarrow \mathbb{H}^{n+1}$ be a graphical hypersurface $F(x)= (x,f(x))$ for some smooth $f:\R^n \rightarrow \R$. The tangent vectors are spanned by $F_i = \partial_i + f_i\partial_{n+1}$. Denote by $\nabla f = (f_1,\dots,f_n)$ the Euclidean gradient and by $\|\nabla f\|$ its Euclidean norm. Then we can choose the upward pointing unit normal 

\begin{equation*}
        N = \frac{f}{\sqrt{1+\|\nabla f\|^2}} \begin{bmatrix}- \nabla f \\ 1\end{bmatrix}
\end{equation*}

The covariant derivatives with respect to the ambient hyperbolic metric are
$$
\nabla_{F_i}F_j = f^{-1}((\delta_{ij}+ff_j -f_if_j) \partial_{n+1} -f_i\partial_j-f_j\partial_i)
$$
which then implies that the second fundamental form is
\begin{equation}\label{eq of A}
    A = \frac{1}{f^2\sqrt{1+\|\nabla f\|^2}}(\mathrm{Id} + \nabla f \otimes \nabla f + f \, \mathrm{Hess}(f))
\end{equation}

where $\mathrm{Hess}(f)$ is the Euclidean Hessian and $[\nabla f \otimes \nabla f]_{ij}=f_if_j$. Thus, for the convexity of the hypersurfaces we only need to investigate the positive definiteness of $A$.
\begin{figure}[h]
	\centering
	\begin{tikzpicture}[scale=2]
		\def\r{2} % radius
		\def\angEl{35} % elevation angle
		\def\angAz{-105} % azimuth angle
		
		\draw [black,ultra thick,domain=30:150] plot ({2*cos(\x)}, {2*sin(\x)-1}) node[above right=0.65cm and 6.5cm]{$S$};
		\fill [ball color=gray,opacity=0.4,domain=30:150] plot ({2*cos(\x)}, {2*sin(\x)-1});
		\draw[red, thick, dashed] (1.725,0) arc (10:170:1.75 and 0.45);
		\fill[ball color=gray,opacity=0.4] (-1.745,0) arc (180:360:1.745 and 0.6);
		\draw[red, thick] (-1.745,0) arc (180:360:1.745 and 0.6);
		
		\draw[thick,-stealth] (0,0,-1.5) -- (0,0,2.5); % z-axis
		\draw[thick,-stealth] (-2.3,0,0) -- (2.5,0,0); % x-axis
		\draw[thick,-stealth] (0,-1.5,0) -- (0,1.5,0) node[right]{$x_{n+1}$}; % y-axis
	\end{tikzpicture}
	\caption{$S$ defined by $f(x)=\sqrt{4q^2-\|x\|^2}-q$}
\end{figure}
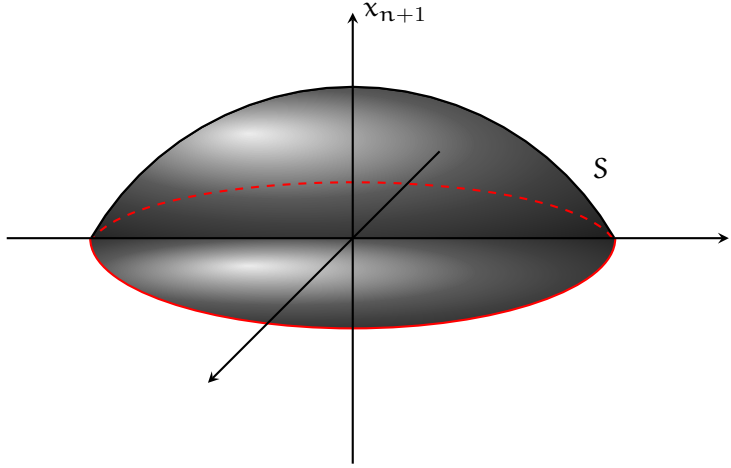

\paragraph{Spheres beyond infinity}
Our aim is to construct a foliation of an open set in $\mathbb{H}^{n+1}$.
Such a foliation and the foliated maximum principle will lead to the \textit{Horosphere theorem}~\ref{thm:horoshphere}.
Such a foliation will consist of graphical hypersurfaces in $\mathbb{H}^{n+1}$ whose convexity can be investigated by making use of Equation~\ref{eq of A}.
Let $f:\R^n \rightarrow \R$ be given by $f(x) = \sqrt{4q^2-\|x\|^2}-q$ for some $q>0$.
Consider the embedding $F:U\to \mathbb{H}^{n+1}$ where $U=\{x\in\R^n\,|\,f(x)>0$\} with $F(x)=(x,f(x))$. 
The defined submanifold is the piece of the Euclidean sphere of radius $2q$ and centre $(0,\dots,0,-q)$ as a subset of $\mathbb{H}^{n+1}$. The gradient is $\nabla f = \frac{-x}{f(x)+q}$ and the Hessian is $$\mathrm{Hess}(f) = -\frac{\mathrm{Id}}{f(x)+q}  - \frac{x \otimes x}{(f(x)+q)^3} $$ where $[x\otimes x]_{ij} = x_ix_j$. A short computation shows that the second fundamental form is
$$
A = \frac{1}{2f(x)} \left( \mathrm{Id} + \frac{x\otimes x}{(f(x)+q)^2}  \right).
$$
The positive definiteness of $A$ follows by noticing that a matrix of the form $\mathrm{Id} + \alpha^2 x\otimes x$ has the eigenvalue $1+ \alpha^2\|x\|^2$ with multiplicity one and the eigenvalue $1$ with multiplicity $n$.

\begin{thm}[Horosphere theorem]\label{thm:horoshphere}
Let $H = \{x\in\mathbb{H}^{n+1}\,\mid\, x_{n+1}=c\}$ for some $c>0$ be a horosphere in the hyperbolic space $\mathbb{H}^{n+1}$. Then every proper harmonic map into the set $\{x\in\mathbb{H}^{n+1}\,\mid\, x_{n+1}>c\}$, which is the region above the horosphere $H$, is constant.
\end{thm}
\begin{proof}
	Let $f$ be a proper harmonic map and $p\in M$ such that $q'=f(p)$ lies above the horosphere $H$.
	By a hyperbolic isometry, we can assume that $q' = (0,\dots,0,q)$ for some $q>c$. Define the foliation $\mathcal{F}$ by the maps $f_t(x) = \sqrt{4t^2-\|x\|^2}-t$ for $t\in [\varepsilon,q]$ for some $0<\varepsilon<c$.
	By the previous discussion, $\mathcal{F}$ is a strictly convex foliation, hence Theorem~\ref{thm_renanito} applies, implying that $f$ needs to intersect the horosphere of height $\varepsilon<c$.	Thus, $f$ can not lie above $H$.
\end{proof}

Let $c>0$ and $f:\R^{n} \rightarrow \R$ be a continuous function such that for all $x\in \R^{n}$ the inequality $-\varepsilon\leq f(x)$ holds for $0< \varepsilon < c$. Then the graph $\mathrm{gr}(f+c)$ in $\R^{n-1}\times (0,\infty) = \mathbb{H}^{n+1}$ defines a perturbed version of the horosphere. The same conclusion holds for this perturbed version as well.

\paragraph{Comments on the Horosphere theorem} We would like to highlight two important theorems that influenced the statement of Theorem~\ref{thm:horoshphere}. The first, by Rodriguez and Rosenberg \cite{rodriguez1998half}, proves a halfspace theorem for mean curvature one surfaces in the hyperbolic space $\mathbb{H}^3$, whose analogy with our result is obvious.

The second, by Mazet~\cite{mazet2013general}, covers a general halfspace theorem for constant mean curvature surfaces in 3-manifolds. This paper explores conditions allowing two surfaces with equal mean curvature to coexist in the same 3-space. Specifically: "If $\Sigma_H$ is a parabolic constant mean curvature (CMC) surface with mean curvature $H$ and any equidistant surface on the non-mean convex side has mean curvature less than $H$, then any CMC $H$ surface on the non-mean convex side of $\Sigma_H$ is an equidistant surface to $\Sigma_H$."
This general result, however, requires the target space to be three-dimensional and the domain to be parabolic, which limits its application in higher dimensions; for example, the Euclidean space $\R^k$ with a flat metric is parabolic if and only if $k=1$ or $k=2$. In contrast, our weaker theorem applies to proper harmonic maps without any restrictions on the dimensions of either the domain or the hyperbolic space target.

\section{A Riemannian cone theorem}\label{RiemannSec}
In this section, we move to the setting of Riemannian manifolds and prove a non-existence result for perturbed Riemannian cones.
Employing again a foliation argument, we will prove that the only proper harmonic maps into $R$ are the constant ones.

\paragraph{Riemannian halfspaces} As a first step, we need a generalised definition of a halfspace in a Riemannian manifold.
\begin{defn}
	Let $(N,h)$ be a complete Riemannian manifold. A \textbf{Riemannian halfspace} at $p\in N$ in the direction $v\in T_pN$ is $$H_{\gamma} := \bigcup_{0 < t} \overline{B_{t}(\gamma(t))}.$$ where $\gamma:[0,\infty)\rightarrow N$ is a unit-speed minimising geodesic ray $\gamma(t) = \exp_p(tv)$ starting at $p$ in the direction of $v$.
	
\end{defn}

Clearly, since the geodesic ray is minimising and of infinite length, $N$ must be non-compact. For example, a torus with an irrational angle line does not satisfy this property. Actually, the completeness is not strictly necessary. Take for example $\R^2\setminus (0,-1)$ with $p=(0,0)$ and $v = (0,1)$. Then we still obtain the usual halfspace although $\R^2\setminus (0,-1)$ is not complete with the Euclidean metric. Another example of a halfspace is a horosphere inside hyperbolic space. For this purpose, take $p=(0,\dots, 0,c) \in \mathbb{H}^n$ and $v = (0,\dots,0,1)$.

\paragraph{Perturbed Riemannian perturbed cones} Perturbed Riemannian cones are going to be an abstraction of the idea of a cone being a subset inside a halfspace defined by a point, a direction and a solid angle.

\begin{defn}\label{defn:riemanniancone}
	Let $\gamma:[0,\infty)\rightarrow N$ be a unit-speed \textbf{minimising geodesic} ray and $r:(0,\infty)\rightarrow (0,\infty)$ a continuous function satisfying $r(t)<\min\{t, r_c(\gamma(t))\}$ where $r_c(p)$ is the convexity radius of $N$ at $p$.
	The subset
	$$
	C_{\gamma,r} := \bigcup_{0<t} \overline{B_{r(t)}(\gamma(t))}.
	$$
	is a \textbf{perturbed Riemannian cone} if the following separation property holds: For every $t' \in (0,\infty)$ the set $C_{\gamma,r} \setminus \overline{B_{r(t)}(\gamma(t'))}$ consists of two disjoint connected components. The function $r$ is referred to as the \textbf{cone radius function}.
\end{defn}

The property $r(t)<t$ assures that the perturbed Riemannian cone $C_{\gamma,r}$ is a subset of the Riemannian halfspace $H_\gamma$ and the property $r(t) < r_c(\gamma(t))$ is necessary for $S_{r(t)}(\gamma(t)):= \partial B_{r(t)}(\gamma(t))$ being strictly convex. Thus, a Riemannian perturbed cone intrinsically possesses a strictly convex foliation by the boundaries of geodesic balls. Note that cones in Euclidean space are a canonical example of a perturbed Riemannian cone, where in addition the radius function $r$ has constant derivative smaller than 1.
\\
The following lemma characterises when geodesic balls along a minimising geodesic rays are contained in each other, i.e. when the separating property of a cone fails. Since the proof relies on the mean value theorem, we will assume differentiability for the cone radius function.

\begin{lemma}
	Let $\gamma:[0,\infty)\rightarrow N$ be a unit-speed minimising geodesic ray and $r:[0,\infty)\rightarrow [0,\infty)$ a continuously differentiable function satisfying $r(0)=0$ and $r(t)<t$. Fix an open subset $(a,b)\subseteq (0,\infty)$. The following are equivalent:
	\begin{enumerate}
		\item $r'\geq1$ on $(a,b)$ and
        \item $B_{t_1}\subseteq B_{t_2}$ for all $a<t_1<t_2<b$ where $B_t = \overline {B_{r(t)}(\gamma(t))}$ is the closed geodesic ball in $N$.
	\end{enumerate} 
\end{lemma}
\begin{proof}
Assuming (a) the mean value theorem implies $t_2-t_1\leq r(t_2)-r(t_1)$. Let $p \in B_{t_1}$, that is $d(p,\gamma(t_1))\leq r(t_1)$. Now
\begin{align*}
    d(p,\gamma(t_2)) &\leq d(p,\gamma(t_1)) + d(\gamma(t_1),\gamma(t_2))
    \\ &\leq r(t_1)+ t_2 - t_1
    \\ &\leq r(t_1) + r(t_2)-r(t_1)
    \\ &= r(t_2).
\end{align*}
Notice that the second inequality follows by $\gamma$ being a unit speed minimising geodesic. Hence, $p\in B(t_2)$ and (b) follows.

Assuming (b) we obtain $d(\gamma(t_2),\gamma(t_1-r(t_1))) \leq r(t_2)$ and by the properties of $\gamma$ we get $d(\gamma(t_2),\gamma(t_1-r(t_1))) = t_2-t_1+r(t_1)$. Rearranging implies that, $$1 \leq \frac{r(t_2) - r(t_1)}{t_2-t_1}$$ and assertion (a) follows.
\end{proof}

\begin{thm}[Perturbed Riemannian cone theorem]\label{thm:cone_theorem}
	Let $C_{\gamma,r}$ be a perturbed Riemannian cone inside the complete Riemannian manifold $(N,h)$. Then every proper harmonic maps into $C_{r,\gamma}\setminus \partial C_{r,\gamma}$ is constant.
\end{thm}

\begin{proof}
	The convexity radius assumption assures that the boundaries of the geodesic balls are strictly convex. Hence, by the above remarks we know that there is a strictly convex foliation present and Theorem~\ref{thm_renanito} applies.
\end{proof}

\paragraph{Applications of the perturbed Riemannian cone theorem} There are many possible applications of Theorem~\ref{thm:cone_theorem} to obtain new cone-type theorems for Riemannian geometry. We point out a couple of interesting ones to show the generality of the theorem in the Riemannian setting.

\begin{enumerate}
	\item Let $r(t) = \cos(\theta)t$ for a fixed angle $\theta \in (0,\pi/2)$ and $\gamma(t)=(0,t)$ with $t\in [0,\infty)$ inside $\R^2$. Then $C_{r,\gamma}$ is a perturbed Riemannian cone and $\partial C_{r,\gamma}$ is a perturbed Euclidean cone. Thus, Theorem~\ref{thm:cone_theorem} recovers the Theorem~\ref{thm:perturbed_wedge} for certain cones.
	\item In the case of hyperbolic spaces Theorem~\ref{thm:cone_theorem} is strictly weaker than Theorem~\ref{thm:horoshphere} since in that case we could prove the theorem for the halfspace, i.e. for the radius function $r(t)=t$.
	\item Let $M = \mathbb{H}^k\times \R^l$ equipped with the product metric. Since the convexity radius of $M$ is infinity, one can take any geodesic ray $\gamma$ and a radius function $r(t) = \cos(\theta)t$ with $\theta \in(0,\pi/2)$.
	\item Consider $M=(S^1)^n\times \R$ as $[-\pi , \pi]^n\times \R$ with opposite sides identified and equipped with the flat metric. The diffeomorphism of $\psi:\R^{n+1}\rightarrow (-\pi,\pi)^n\times \R$ given by $$(x_1,\dots,x_n,s) \mapsto (2\arctan(x_1),\dots,2\arctan(x_n),s)$$ allows the definition of Riemannian perturbed cones in $M$ via taking the image of a cone inside $\R^{n+1}$. For instance, let $C_{\gamma,r}\subseteq \R^{n+1}$ be defined by $\gamma(t)=(0,\dots,0,t)$ with $t\in [0,\infty)$ and $r(t) = \cos(\theta) t$ for $\theta \in (0,\pi/2)$. Then under $\psi$ the cone gets mapped to $C_{\gamma,R}$ inside $M$ where $\gamma$ is just like before and the transformed radius function is $R(t) = 2 \arctan(r(t)) = \arctan( \cos(\theta) t)$. 
	\item Let $M = S^n \times \R$ be equipped with the product metric. Denote by $S$ the south pole of $S^n$ and by $\gamma(t) = (S,t)$ with $t\in [0,\infty)$ the minimising geodesic. Since the convexity radius of $S^n$ is $\pi/2$ we need to stay below that value. Again, we can take the radius function $r(t)=\arctan(t)$ and define a Riemannian perturbed cone inside $M$.
\end{enumerate}

\bibliography{perturbedconeharmonic}
\addcontentsline{toc}{section}{\bibname}

\end{document}